\documentclass[11pt,a4paper]{article}

\usepackage{latexsym,epsf}
\usepackage{amsthm,amssymb,amsfonts,epsfig,color,amsmath,bbm}
\usepackage{wrapfig}

\usepackage{graphicx}
\usepackage{url}

\newtheorem{theorem}{Theorem}[section]
\newtheorem{lemma}[theorem]{Lemma}
\newtheorem{corollary}[theorem]{Corollary}
\newtheorem{claim}{Claim}
\newcommand{\claimqed}{\hspace*{\fill} $\triangle$  \ifmmode \else
    \par\addvspace\topsep\fi}
\newenvironment {claimproof}{\par\addvspace\topsep\noindent{\it Proof.}
    \ignorespaces \rm}{\claimqed}

\def\R{\mathbb{R}}
\begin{document}

\title{On the Fiedler value of large planar graphs}

\author{
      \normalsize Lali Barri\`ere\\
      \small\sf Departament de Matem\`atica Aplicada IV,\\
      \small\sf Universitat Polit\`ecnica de Catalunya\\[-1mm]
      \small\sf {\tt lali@ma4.upc.edu}
\and
      \normalsize Clemens Huemer\\
      \small\sf Departament de Matem\`atica Aplicada IV,\\
      \small\sf Universitat Polit\`ecnica de Catalunya\\[-1mm]
      \small\sf {\tt clemens@ma4.upc.edu}
\and
      \normalsize Dieter Mitsche\\
      \small\sf Departament de Matem\`atica Aplicada IV,\\
      \small\sf Universitat Polit\`ecnica de Catalunya\\[-1mm]
      \small\sf {\tt dieter.mitsche@ma4.upc.edu}
\and
      \normalsize David Orden\\
      \small\sf Departamento de F\'{\i}sica y Matem\'aticas\\
      \small\sf Universidad de Alcal\'a \\[-1mm]
      \small\sf {\tt david.orden@uah.es}
}

\date{}
\maketitle

\begin{abstract}
The Fiedler value $\lambda_2$, also known as algebraic connectivity, is the second smallest Laplacian eigenvalue of a graph.
We study the maximum Fiedler value among all planar graphs $G$ with $n$ vertices, denoted by $\lambda_{2\max}$,
and we show the bounds $2+\Theta(\frac{1}{n^2}) \leq \lambda_{2\max} \leq 2+O(\frac{1}{n})$.
We also provide bounds on the maximum Fiedler value for the following classes of planar graphs:
Bipartite planar graphs, bipartite planar graphs with minimum vertex degree~3, and outerplanar graphs.
Furthermore, we derive almost tight bounds on $\lambda_{2\max}$ for two more classes of graphs, those of bounded genus and $K_h$-minor-free graphs.
\end{abstract}

\section{Introduction}

Let $G=(V,E)$ be a simple graph with vertex set $V=\{v_1,\ldots,v_n\}$.
The \emph{Laplacian matrix} $L(G)$ is the $n \times n$ matrix with entries
$$
 \ell_{i,j} = \left\{ \begin{array}{cl}
         deg(v_i) & \mbox{if $i=j$},\\
        -1 & \mbox{if $i \neq j$ and $v_iv_j \in E$},\\
        0 & \mbox{if $i \neq j$ and $v_iv_j \notin E$}.\end{array} \right.
$$
Let the eigenvalues of $L(G)$ be $0=\lambda_1 \leq \lambda_2 \leq \lambda_3 \leq \cdots \leq \lambda_n$.
The second smallest eigenvalue $\lambda_2$, or $\lambda_2(G)$, is called the {\emph{Fiedler value}} or  {\emph{algebraic connectivity}}~\cite{fiedler} of~$G$.
It is related to a number of graph invariants and it plays a special role in many problems in Physics and Chemistry, where spectral techniques can be applied~\cite{abreu,fiedler,mohar91,mohar}.
Another classical problem for which the techniques introduced in~\cite{fiedler} have revealed to be very successful is graph partitioning~\cite{elsner}. The Fiedler value has also been proved to be related to the size of separators, as well as to the quality of geometric embeddings of the graph~\cite{ST96,ST07}.

A number of results have been obtained for $\lambda_2$, for which we refer the interested reader to the surveys~\cite{abreu,mohar}. As for recent works, the authors of~\cite{BLR10} make use of flows and the choice of an appropriate metric for proving bounds on $\lambda_2$. Similar techniques are used in~\cite{KLPT09} to study higher eigenvalues of graphs of bounded genus. Another work devoted to upper bounds on the algebraic connectivity is~\cite{freitas}.

The main goal of the present work is to study the maximum of $\lambda_2(G)$ over all planar graphs $G$ with $n$ vertices, which will be denoted as $\lambda_{2\max}.$ The bound $\lambda_{2\max}<6$ follows easily, since for any graph $G=(V,E)$ with $n$ vertices $\lambda_2(G)\leq \frac{2|E|}{n-1}$~\cite{fiedler} and if $G$ is planar then $|E| \leq 3n-6$. Molitierno~\cite{molitierno} proved that $\lambda_{2\max}\leq 4$, with exactly two planar graphs attaining this bound: the complete graph with four vertices, $K_4$, and the octahedral graph $K_{2,2,2}$. It is known that $\lambda_2$ is much smaller for some planar graph classes.
In particular, trees have $\lambda_2\le 1$, with the bound achieved only for $K_{1,n-1}$~\cite{merris}.
Moreover, Spielman and Teng~\cite{ST07} proved that for the class of bounded-degree planar graphs with $n$ vertices, $\lambda_{2\max}$ tends towards zero when $n$ tends towards infinity.

We also study $\lambda_{2\max}$ for bipartite planar graphs, outerplanar graphs, graphs of bounded genus and $K_h$-minor-free graphs. Table~\ref{tab:results} summarizes our results. Some of them improve our previous results presented in~\cite{eurocomb}.
\begin{table}[h!]
\renewcommand*\arraystretch{2.2} 
\begin{center}
  \begin{tabular}{| l |  c | }
    \hline
    Planar graphs & $ 2 + \Theta\left(\frac{1}{n^2}\right) \leq \lambda_{2\max} \leq 2 +  O\left(\frac{1}{n}\right)$  \\ \hline
    Bipartite planar graphs, $\delta=3$
    & $ 1 + \Theta\left(\frac{1}{n^2}\right) \leq \lambda_{2\max} \leq 1 +  O\left(\frac{1}{n^{1/3}}\right)$  \\ \hline
    Bipartite planar graphs, $n$ large & $\lambda_{2\max}=2$  \\  \hline
    Outerplanar graphs & $ 1 + \Theta\left(\frac{1}{n^2}\right) \leq \lambda_{2\max} \leq 1 +  O\left(\frac{1}{n}\right)$  \\  \hline
    Graphs of bounded genus~$g$ & $2 + \Theta\left(\frac{1}{n^2}\right) \leq  \lambda_{2\max} \leq 2+O\left(\frac{1}{\sqrt{n}}\right)$  \\  \hline
    $K_h$-minor-free graphs,  $4 \leq h\leq 9$ & $h-2 \leq \lambda_{2\max} \leq h-2+O\left(\frac{1}{\sqrt{n}}\right)$  \\  \hline
    $K_h$-minor-free graphs, $h$ large & $ \lambda_{2\max} \leq \alpha h\sqrt{\log(h)}+ O\left({\frac{ \alpha h^{5/2}\sqrt{\log(h)}}{\sqrt{n}}}\right)$\\
    & for $\alpha = 0.319\ldots + o(1)$ \\
    \hline
  \end{tabular}
\end{center}
\caption{The bounds on $\lambda_{2\max}$ obtained for each class of graphs studied.}
\label{tab:results}
\end{table}

For all upper bounds on $\lambda_{2\max}$ we make use of the following embedding lemma, which makes clear the relation between geometric embeddings of graphs and the Fiedler value. It is a direct consequence of the so-called Courant-Fischer principle and can be found in~\cite{mohar91,ST07}:
\begin{lemma}[Embedding Lemma]
\label{lem:embedding}
Let $G=(V,E)$ be a graph. Then $\lambda_2$, the Fiedler value of $G$, is given by
$$ \lambda_2 = \min \frac{\sum_{(i,j) \in E} ||\vec{v_i} - \vec{v_j} ||^2 } {\sum_{i=1}^{n}{||\vec{v_i}||^2}}$$
where the minimum is taken over all non-zero vectors $\{ \vec{v_1},\cdots, \vec{v_n}\} \subset \mathbb{R}^n$ such that $\sum_{i=1}^{n} \vec{v_i} =\vec{0}$.
\end{lemma}

We will make use of the Embedding Lemma~\ref{lem:embedding} in two ways. Before introducing them, let us state a result by Spielman and Teng for planar graphs:

\begin{theorem}[Spielman-Teng~\cite{ST07}]
\label{thm:spielman-teng}
Let $G$ be a planar graph with $n$ vertices and maximum degree~$\Delta$. Then, the Fiedler value of $G$ is at most $\frac{8\Delta}{n}.$
\end{theorem}

In their proof, Spielman and Teng first used Koebe's kissing disk embedding~\cite{koebe} on the plane and then mapped the points of this embedding onto the unit sphere, using stereographic projection and sphere-preserving maps, in such a way that $$\sum_{(i,j) \in E} ||\vec{v_i} - \vec{v_j} ||^2 \leq 8\Delta \qquad\mbox{and}\qquad \sum_{i=1}^{n} \vec{v_i} =\vec{0}.$$ Their result is then straightforward from the Embedding Lemma~\ref{lem:embedding}.\\

Our first technique, for general planar graphs, uses the Embedding Lemma together with Theorem~\ref{thm:spielman-teng}. We embed vertices of high degree in the origin and use the embedding of Spielman and Teng for the remaining graph of bounded vertex-degree. For example, an optimal embedding (that gives the exact value of $\lambda_2$) for the wheel graph $W_{n+1}$ is to place the vertex of degree~$n$ at the origin and the remaining vertices on the unit circle, as vertices of a regular $n$-gon. Examples of planar graphs on $n$ vertices with large Fiedler value are constructed similarly, as we will see in Section~\ref{sec:lower_planar}.\\

Our second technique uses the Embedding Lemma together with a separator. We recall that a \emph{separator} $X$ of a graph $G=(V,E)$ is a subset $X \subset V$ whose removal from $G$ breaks the graph into several connected components. Similarly to the first method, we place the separator $X$ at the origin and the remaining vertices on the unit circle. This gives us the following bound on the Fiedler value~$\lambda_2(G)$, whose proof is deferred to Section~\ref{sec:separator}.

\begin{theorem}\label{thm:separator}
Let $G$ be a graph on $n$ vertices which has a separator $X$ such that each connected component of $G-X$ has at most $\frac{n-|X|}{2}$ vertices. Then $$\lambda_2(G) \leq \frac{|E_{X,G-X}|}{n-|X|},$$ where $E_{X,G-X}$ is the set of edges of $G$ with one endpoint in $X$ and the other in $G-X$.
\end{theorem}

Section~\ref{sec:upper} is devoted to prove the upper bounds on $\lambda_{2\max}$ given in Table~\ref{tab:results}. We outline here the main ingredients for each case:
\begin{itemize}
  \item The upper bound for planar graphs is obtained by an appropriate combination of our two techniques.
  \item The bound for the class of bipartite planar graphs and for $n$ large is obtained as a corollary of a result for bipartite planar graphs with minimum vertex degree~$3$, which is based only on Theorem~\ref{thm:spielman-teng}. A difficulty that arises here is how to bound the number of edges connecting vertices of high degree with vertices of small degree.
  \item For the class of outerplanar graphs we only need Theorem~\ref{thm:separator}.
  \item For the class of graphs of bounded genus, we use Theorem~\ref{thm:separator} together with the separator theorem of Gilbert et al.~\cite{gilbert}.
  \item For the class of $K_h$-minor-free graphs, the separator theorem for non-planar graphs by Alon et al.~\cite{alon}, together with the known maximal number of edges in $K_h$-minor-free graphs~\cite{song,thomason} can be used in Theorem~\ref{thm:separator}.
\end{itemize}

Finally, in Section~\ref{sec:lower_planar} we give examples of constructions which attain the lower bounds on $\lambda_{2\max}$ given in Table~\ref{tab:results}.

\section{Proof of Theorem~\ref{thm:separator}}\label{sec:separator}

In this section we prove two lemmas which together with the Embedding Lemma~\ref{lem:embedding} imply Theorem~\ref{thm:separator}.

\begin{lemma}
\label{lem:geometric}
Let $a,b,c>0$ be real numbers such that $a< b+c$, $b<c+a$, and $c<a+b$. Then, there are three points $\vec{w_1},\vec{w_2},\vec{w_3}\in\R^2$ on the unit circle,
such that $a\cdot \vec{w}_1 + b\cdot \vec{w}_2 + c\cdot \vec{w}_3=\vec{0}$.
\end{lemma}

\begin{proof}
Consider the triangle with side lengths $a,b,c$ and respective opposite angles $\alpha,\beta,\gamma$ (see Figure~\ref{fig:trigon}) and define
$$
\begin{array}{ccc}
\vec{w}_1=(1,0), & \vec{w}_2=(\cos(\pi+\gamma),\sin(\pi+\gamma)), & \vec{w}_3=(\cos(\pi-\beta),\sin(\pi-\beta)).
\end{array}
$$

\begin{figure}[htb]
\begin{center}
\includegraphics[width=0.25\textwidth]{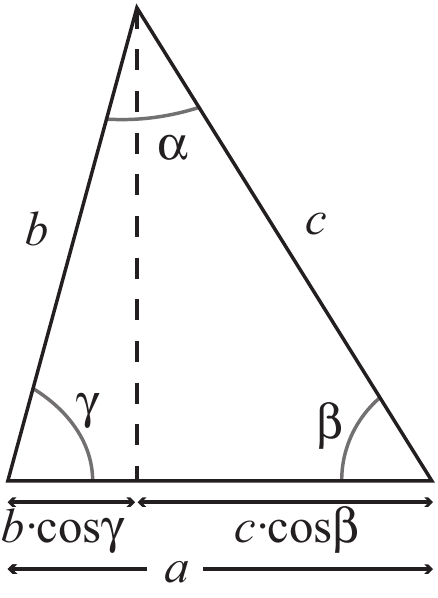}
\caption{Illustration of the proof of Lemma~\ref{lem:geometric}.}
\label{fig:trigon}
\end{center}
\end{figure}

Then $a\cdot \vec{w}_1+b\cdot \vec{w}_2+c\cdot \vec{w}_3$ equals
$$
\begin{array}{l}
(a,0)+(b\cos(\pi+\gamma),b\sin(\pi+\gamma))+(c\cos(\pi-\beta),c\sin(\pi-\beta))
\\
=(a,0)+(-b\cos\gamma,-b\sin\gamma)+(-c\cos\beta,c\sin\beta)
\\
=(a-b\cos\gamma-c\cos\beta,c\sin\beta-b\sin\gamma).
\end{array}
$$

The first component of this vector is~$0$, since $a=c\cos\beta+b\cos\gamma$ (see Figure~\ref{fig:trigon}). The second component is also~$0$ since, by the law of sines,
$
\frac{\sin\alpha}{a}=\frac{\sin\beta}{b}=\frac{\sin\gamma}{c}.
$
\end{proof}

\begin{lemma}\label{lem:embed}
Let $r,s$ be integers such that $s\geq r \geq 2$. Let $k_1,\ldots,k_r>0$ be real numbers such that $k_1+k_2+\dots +k_r=s$ and $k_i\le \frac{s}{2},\ \forall i\in\{1,\ldots,r\}$. Then, there are $r$ points $\vec{v}_i=(x_i,y_i)\in\R^2$ (not necessarily different) on the unit circle such that
$$
\sum_{i=1}^r k_i\cdot \vec{v}_i=\vec{0}.
$$
\end{lemma}

\begin{proof}
Without loss of generality, we can assume that $k_1\le k_2\le \dots\le k_r$.
Since $k_i\le \frac{s}{2}$ for all $i\in\{1,\ldots,r\}$, we have $k_1+k_2+\cdots +k_{r-1}\ge k_r$.
Let $\ell\in\{1,\ldots,r\}$ be the smallest integer such that $k_1+k_2+\dots +k_{\ell}\ge k_{\ell+1}+\dots +k_r$.

\begin{itemize}
\item
If $k_1+k_2+\cdots +k_{\ell}= k_{\ell+1}+\cdots +k_r=\frac{s}{2}$, then it is enough to define $\vec{v}_i=(1,0)$ for $i\in\{1,2,\ldots, \ell\}$ and $\vec{v}_i=(-1,0)$ for $i\in\{\ell+1,\ldots,r\}$, since then
$$
\sum_{i=1}^r k_i\cdot \vec{v}_i=
\sum_{i=1}^{\ell} k_i \cdot (1,0)+\sum_{i=\ell+1}^{r} k_i\cdot (-1,0)=
\frac{s}{2}(1,0)+\frac{s}{2}(-1,0)=
\vec{0}.
$$
\item
Otherwise, $k_1+k_2+\dots +k_{\ell}> k_{\ell+1}+\dots +k_r$ and also $k_1+k_2+\dots +k_{\ell-1}< k_{\ell}+\dots +k_r$.
Then, we take $a=k_1+\dots +k_{\ell-1}$, $b=k_{\ell}$, and $c=k_{\ell+1}+\dots +k_r$, which fulfill that
$a<b+c$, $b<c+a$, and $c<a+b$.

Hence, by Lemma~\ref{lem:geometric}, there are three points $\vec{w_1}$, $\vec{w_2}$, and $\vec{w_3}$ on the unit circle such that $a\cdot \vec{w}_1 + b\cdot \vec{w}_2 + c\cdot \vec{w}_3=\vec{0}$.

Let us now define

\begin{center}
$\vec{v}_i=\vec{w}_1$ for $i\in\{1,\ldots ,\ell-1\}$, \quad $\vec{v}_{\ell}=\vec{w}_2$, \quad $\vec{v}_i=\vec{w}_3$ for $i\in\{\ell+1,\ldots, r\}$,
\end{center}

which satisfy
$$
\sum_{i=1}^r k_i\cdot \vec{v_i}=
(k_1+\cdots +k_{\ell-1})\cdot \vec{w}_1+k_{\ell}\cdot \vec{w}_2+(k_{\ell+1}+ \cdots +k_r)\cdot \vec{w}_3
$$
$$
= a\cdot \vec{w}_1+b\cdot \vec{w}_2+c\cdot \vec{w}_3
=\vec{0}.
$$
\end{itemize}
\end{proof}

\noindent{\bf{Proof of Theorem~\ref{thm:separator}.}}
Use Lemma~\ref{lem:embed}, with $k_1,\ldots,k_r$ being the sizes of the connected components of~$G-X$ and $s=n-|X|$, to obtain a point~$\vec{v_i}$ on the unit circle for each connected component, with $\sum_{i=1}^r k_i\cdot \vec{v}_i=\vec{0}$. Place the vertices of each component at the corresponding~$\vec{v_i}$ and the vertices of $X$ at the origin. Then $\sum_{i=1}^n \vec{v}_i=\vec{0}$ and each edge between $X$ and $G-X$ has length one, while the remaining ones have length zero. The Embedding Lemma~\ref{lem:embedding} gives then $\lambda_2(G) \leq \frac{|E_{X,G-X}|}{n-|X|}$, as desired.
\qed

\section{Upper bounds on $\lambda_{2\max}$}
\label{sec:upper}

\subsection{General planar graphs}
\label{subsec:planar}
In our conference paper~\cite{eurocomb} we used the embedding lemma of Spielman and Teng to obtain that, for the class of planar graphs, $\lambda_{2\max}\leq 2+O(1/\sqrt{n})$. In the present paper we improve this bound to $\lambda_{2\max}\leq 2+O(1/n)$. The proof relies on the following technical lemma:

\begin{lemma}\label{lem:sep_planar}
Let $G$ be a planar graph on $n$ vertices such that the number of edges incident to vertices of degree at least $\frac{n}{K}$ is at least $(2-\frac{8}{K})n$, for some sufficiently large constant $K$.
Then $G$ has a separator $X$ of constant size such that:
\begin{itemize}
  \item[(i)] Each component of $G-X$ has less than $\frac{n}{3}$ vertices.
  \item[(ii)] The cardinality of the set~$E_{X,G-X}$ of edges of $G$ with one endpoint in $X$ and the other in $G-X$ is at most $\left(2+\frac{16}{K}\right)n$ plus some constant.
\end{itemize}
\end{lemma}
\begin{proof}
Since adding edges to a graph can only increase the minimal size of the separator and the size of~$E_{X,G-X}$, we can assume that $G$ is a triangulation (maximal planar graph), then having $3n-6$ edges. Consider a plane embedding of $G$.
Denote by $\cal{C}$ the set of vertices of $G$ with degree at least $\frac{n}{K}$ and note that for every planar graph $G$, it holds that $|\cal{C}|$ is in $O(1)$. Denote by $E_{\cal{C}}$ the set of edges emanating from $\cal{C}$.

The vertices of $\cal{C}$ can be seen as centers of wheels. Observe that for each wheel the length of its bounding cycle equals the number of edges incident to its center. Thus, let us rephrase the hypothesis of the lemma in terms of edges of those bounding cycles, which we will call \emph{boundary edges}. Since every edge of a planar graph belongs to at most two triangular faces, any edge is a boundary edge of at most two wheels.
Let $B_i$, with $i\in\{1,2\}$, be the set of boundary edges of the graph that belong to $i$ wheels with center in~$\cal{C}$.

Therefore, the hypothesis $|E_{\cal{C}}| \geq \left(2-\frac{8}{K}\right)n$ can be rephrased as
\begin{equation}
|B_1|+2|B_2|\geq\left(2-\frac{8}{K}\right)n.
\label{eq:planarB_1+2B_2}
\end{equation}

Now, let us count the number of edges in $E_{\cal{C}}\cup B_1\cup B_2$. It might happen that two vertices $v$ and $w$ of $\cal{C}$ are adjacent, in which case two boundary edges of the wheel with center $v$ are incident to $w$, and thus belong to the set $E_{\cal{C}}$; and vice versa for two boundary edges of $w$. Since $|\cal{C}|$ is a constant, it follows that the number of edges that are counted both for $E_{\cal{C}}$ and for $B_1\cup B_2$ is a constant, say~$f$. Note that we do not care about the precise value of $f$ (and some more constants defined later), which is not relevant for the proof.
Given that the number of edges of $G$ is $3n-6$, we get that
$$\left(2-\frac{8}{K}\right)n + |B_1|+|B_2| - f \leq |E_{\cal{C}}\cup B_1\cup B_2|\leq 3n-6$$ and thus, from Inequality~(\ref{eq:planarB_1+2B_2}),
$$\left(2-\frac{8}{K}\right)n + \left(2-\frac{8}{K}\right)n-2|B_2|+|B_2| - f \leq 3n-6$$ which gives
\begin{equation}
|B_2|\geq \left(1-\frac{16}{K}\right)n+6-f.
\label{eq:planarB_2}
\end{equation}
We will define a separator $X$ for $G$ such that each component of $G-X$ will have at most $\frac{n}{3}$ vertices.
The separator $X$ will be composed of all vertices of $\cal{C}$ (a constant number) plus a constant number of vertices of $V_{B_2}$, where $V_{B_2}$ denotes the set of endpoints of edges in~$B_2$. In order to choose those vertices of~$V_{B_2}$, we need to define the notion of diamond as follows.

Consider two wheels, with centers $v,w \in \cal{C}$ having common boundary edges. Order these edges cyclically around $v$, so that their endpoints form an ordered list of vertices $a_1,a_2,\ldots, a_\ell$, without repetition.
Take two vertices $a_i,a_j$ with $i < j$ and consider the smallest cycle passing through $v,a_i,w,a_j$ and having all $a_k$ with $i<k<j$ in its interior. Note that the length of this cycle can only be~$4$.
Then, we define as $D_{v,a_i,w,a_j}$ the subgraph of~$G$ induced by the vertices on the boundary and in the interior of the cycle. Such a $D_{v,a_i,w,a_j}$ is said to be a \emph{diamond} if among its vertices only $v$ and $w$ belong to $\cal{C}$.

\begin{claim}
The set of boundary edges $B_2$ can be partitioned into a constant number of diamonds.
\begin{claimproof}
Recall that each edge of $B_2$ belongs to two wheels. Since the number of wheels with center in $\cal{C}$ is constant, so is the number of pairs of wheels with centers in $\cal{C}$. Given such a pair with centers $v$ and $w$, we consider their common boundary edges $a_1a_2=e_1,e_2,\ldots,e_k=a_{\ell-1}a_{\ell}$ and the subgraph bounded by $v,a_1,w,a_\ell$. This might not be a diamond, if other vertices of $\cal{C}$ lie in the interior of the subgraph. In that case, the set of edges $\{e_1,\ldots,e_k\}$ would not be contained in a single diamond but in the union of several diamonds (see Figure~\ref{fig:diamant}).
\begin{figure}[htb]
\begin{center}
\includegraphics[width=0.6\textwidth]{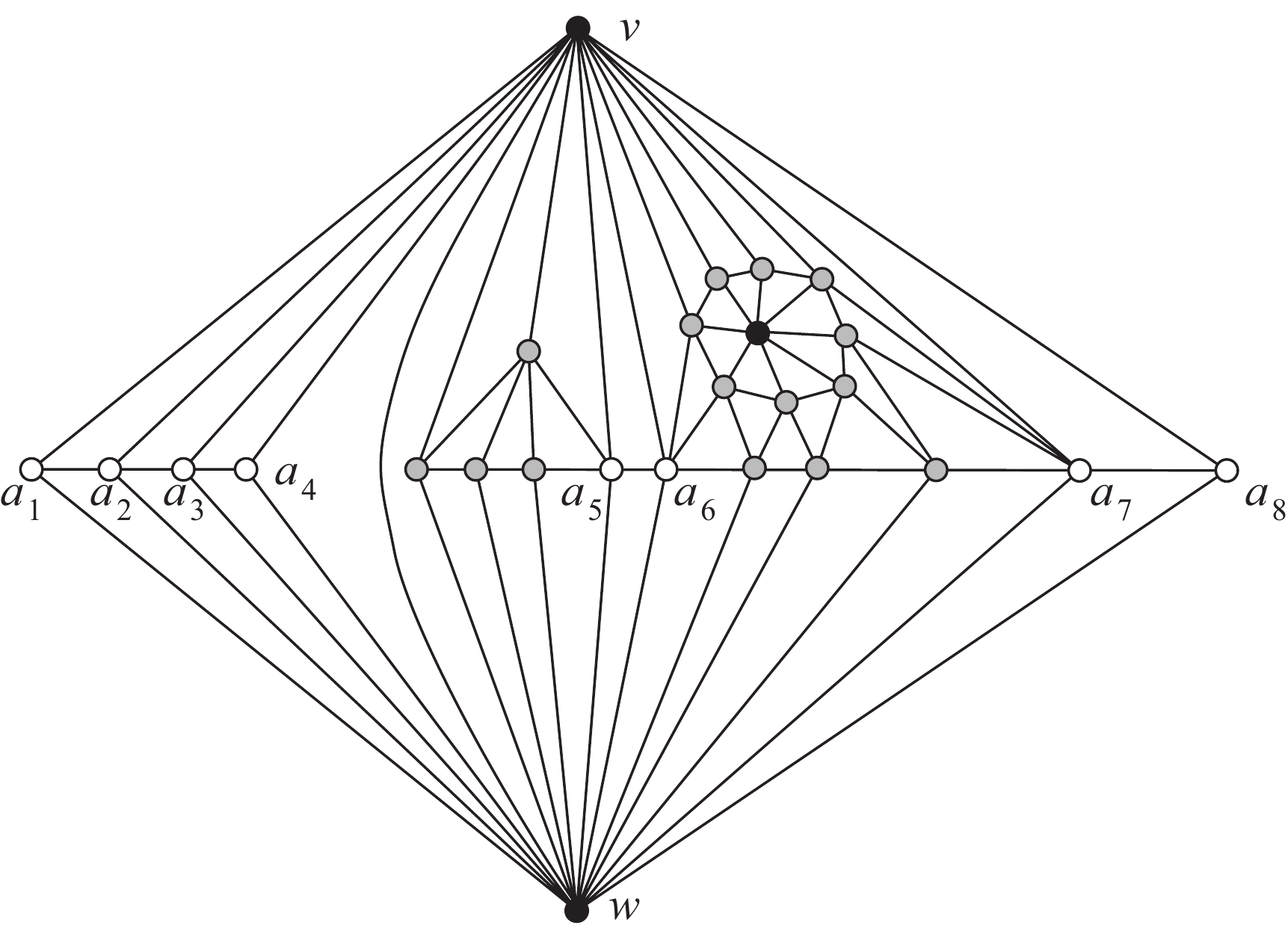}
\caption{Union of several diamonds. Black vertices are in~$\cal{C}$.}
\label{fig:diamant}
\end{center}
\end{figure}
In that case, we can split those edges into at most $|\cal{C}|$ diamonds, which is a constant number.
\end{claimproof}
\end{claim}

\begin{claim}
Each diamond $D$ containing more than $\frac{n}{3}$ vertices has a separator $X_D$ of constant size such that the components of $D-X_D$ have less than $\frac{n}{3}$ vertices each.
\begin{claimproof}
Note that a vertex of~$V_{B_2}$ might belong to several
bounding cycles of wheels and that the number of pairs of those cycles is constant (because of the number of pairs of wheels with centers in~$\cal{C}$ being constant). Hence, from the definition of~$V_{B_2}$ and Inequality~(\ref{eq:planarB_2}), it follows that
\begin{equation}
\label{eq:planarB}
|V_{B_2}|\geq \left(1-\frac{16}{K}\right)n-\bar{f}
\end{equation}
for some constant~$\bar{f}$.

Therefore, the number of vertices that are not in~$V_{B_2}$
is at most $\frac{16}{K}n+\bar{f}$ and so is the number of vertices inside $D$
and not in~$V_{B_2}$.

Hence, it is sufficient to split into four equal components the vertices of~$V_{B_2}$ in $D$,
since then each component can have at most~$\frac{n}{4}+\frac{16}{K}n+\bar{f}\leq\frac{n}{3}$ vertices, where the last inequality holds for large enough values of~$K$.

It just remains to define a separator $X_D$, of constant size, which splits the vertices of~$V_{B_2}$ in $D$
into four equal components.
Let the diamond $D$ be actually $D_{v,a_i,w,a_j}$ and let us call $z_1,z_2,\ldots,z_t$ the edges of~$B_2$ in the diamond, in the cyclic order around~$v$. Then it is enough to include in~$X_D$ the endpoints of the edges $z_{\left\lceil t/4\right\rceil},z_{\left\lceil t/2\right\rceil}$ and $z_{\left\lceil 3t/4\right\rceil}$.
\end{claimproof}
\end{claim}

Now, we define the separator~$X$ as the union of $\mathcal{C}$ and the separators~$X_D$ for the diamonds~$D$ from Claim~1, which are a constant number. Claim~2 ensures that each component of those $D-X_D$ has less than $\frac{n}{3}$ vertices.

 Let us now verify that also each remaining component of~$G-X$ has at most $\frac{n}{3}$ vertices. From Inequality~(\ref{eq:planarB}), we know that the number of vertices not in~$V_{B_2}$ is at most $\frac{16}{K}n+\bar{f}$. Hence, the number of vertices not belonging to any of the diamonds from Claim~1 is at most $\frac{16}{K}n+\bar{f}$. Therefore, the number of vertices in the remaining components of $G-X$ is also at most $\frac{16}{K}n+\bar{f}$. For $K$ sufficiently large, this number is less than $\frac{n}{3}$, which proves part~(i) of the statement.

In order to prove part~(ii) of the statement, we have to bound the cardinality of the set~$E_{X,G-X}$ of edges with one endpoint in~$X$ and the other in~$G-X$. In order to do so, let us look at the complement of~$E_{X,G-X}$ and show that it contains many of the edges in~$B_2$.

Recall that~$B_2$ is a union of cycles and paths, a subset of the bounding cycles of wheels with center in~$\mathcal{C}$, and
that the number of pairs of wheels with centers in $\cal{C}$ is constant. Hence, each vertex in $V_{B_2}$ belongs to a constant number of those cycles and paths that form~$B_2$. Then, by the definition of~$X$, each vertex in~$X$ can only be incident to a constant number of edges in~$B_2$. Furthermore, the separator~$X$ has constant size. It follows that the number of edges in~$B_2\cap E_{X,G-X}$ is some constant~$f'$.

Since~$G$ has~$3n-6$ edges, we get $|E_{X,G-X}|\leq 3n-6 -|B_2|+f'$ and using Inequality~(\ref{eq:planarB_2}) it follows that
\[
|E_{X,G-X}|\leq 3n-6-\left(1-\frac{16}{K}\right)n-6+f+f'=\left(2+\frac{16}{K}\right)n-12+f+f',
\]
which settles part~(ii) of the statement.
\end{proof}

\begin{theorem}\label{thm:upper_planar}
For the class of planar graphs, $\lambda_{2\max} \leq 2+O(\frac{1}{n}).$
\end{theorem}
\begin{proof}
Let $G=(V,E)$ be a planar graph on $n$ vertices and $K$ a sufficiently large constant such that Lemma~\ref{lem:sep_planar} holds. As above, denote by $\cal{C}$ the set of vertices of~$G$ with degree at least $\frac{n}{K}$ and denote by $E_{\cal{C}}$ the set of edges emanating from $\cal{C}$.
We distinguish two cases:\\

\noindent\textit{Case 1:} {$|E_{\cal{C}}| \leq (2-\frac{8}{K})n$}.
Consider the induced subgraph $G'$ of $G$ that is obtained from $G$ by removing all vertices of $\cal{C}$.
Take the embedding of Spielman and Teng for $G'$ and extend it to an embedding for $G$ by placing the vertices of $\cal{C}$ at the origin.
Now use Lemma~\ref{lem:embedding}, together with the number of vertices of $G'$ being $n-|\cal{C}|$.
We get that $\lambda_2(G)$ equals
$$\min_{\sum_{i=1}^{n} \vec{v_i} =\vec{0}} \frac{\sum_{(i,j) \in E} ||\vec{v_i} - \vec{v_j} ||^2 } {\sum_{i=1}^{n}{||\vec{v_i}||^2}}
\leq  \frac{\sum_{(i,j) \in E(G')} ||\vec{v_i} - \vec{v_j} ||^2} {n-|\cal{C}|} + \frac{\sum_{(i,j) \in E_{G,G'}} ||\vec{v_i} - \vec{v_j} ||^2 } {n-|\cal{C}|},$$
where $E(G')$ denotes the set of edges of $G'$ and $E_{G,G'}$ the set of edges connecting a vertex in~$G$ and a vertex in~$G'$.
Then, as in the proof of Theorem~\ref{thm:spielman-teng}, the first summand is at most $\frac{{\frac{8n}{K}}}{n-|\cal{C}|}$. For the second summand, observe that the graph $G=(V,E_{G,G'})$ by assumption has at most $|E_{\cal{C}}| \leq (2-\frac{8}{K})n$ edges, and these in our case have length one, so we get $$\lambda_2(G) \leq \frac{\frac{8n}{K}}{n-|\cal{C}|} + \frac{(2-\frac{8}{K})n}{n-|\cal{C}|}= \frac{2n}{n-|\cal{C}|}.$$ Since $|\cal{C}|$=$O(1)$, this last value yields $2+O(\frac{1}{n}),$ and this case is settled.\\

\noindent\textit{Case 2:} {$|E_{\cal{C}}| > (2-\frac{8}{K})n$}.
By Lemma~\ref{lem:sep_planar}, $G$ has a separator $X$ of constant size such that each component of $G-X$ has at most  $\frac{n}{3} \leq \frac{n-|X|}{2}$ vertices. We can then use Theorem~\ref{thm:separator} and the fact that
the graph $G''=(V,E_{X,G-X})$ is bipartite and planar, hence it has at most $2n-4$ edges. Thus, we get $\lambda_2(G) \leq \frac{2n-4}{n-|X|}$ which is $2+O(\frac1n)$ since $|X|$ is in $O(1)$.
\end{proof}

\subsection{Bipartite planar graphs}

We will first focus on the class of bipartite planar graphs with minimum vertex degree~$3$.
The smallest graph in this class is the $3$-cube, which has $\lambda_2=2$~\cite{fiedler}. Our result shows that for this class $\lambda_{2\max}$ is asymptotically close to~$1$ (with respect to the number of vertices). We will apply techniques similar to those in Case~$1$ of the proof of Theorem~\ref{thm:upper_planar}. However, an issue arising here is to count the number of edges connecting vertices of high degree with vertices of low degree. This is what makes not possible to just mimic the proof of Theorem~\ref{thm:upper_planar}.

Our upper bound on $\lambda_{2\max}$ for bipartite planar graphs with minimum vertex-degree~$3$ relies on the following lemma, which is trivially true for small values of~$k$ but will be used for $k\in\Omega(\sqrt{n})$.

\begin{lemma}\label{lem:bipartiteEdges}
Let $G$ be a maximal bipartite planar graph with $n$ vertices and minimum vertex-degree~$3$. Let $A$ be the set of vertices of $G$ with degree at least~$k$. Then the cardinality of the set $E_{A,G-A}$ of edges of $G$ with one endpoint in~$A$ and the other in~$G-A$ is at most
$$n + \left(\frac{4n-8}{k}\right)^2 + \frac{8n-16}{k}-8.$$
\end{lemma}

\begin{proof}
First, note that maximal bipartite planar graphs on $n\geq 4$ vertices are quadrangulations and have $2n-4$ edges.
We can assume that $\sum_{v \in A} \deg{(v)} \geq n$, as otherwise there are less than~$n$ edges between $A$ and $G-A$.
Call a vertex $v \in G$ a \emph{high-degree vertex} if $\deg{(v)} \geq k$, and a \emph{low-degree vertex} otherwise.
Then $A$ is the set of vertices of $G$ of high degree and $G-A$ that of vertices of low degree.

We are going to show the following bound on the cardinality of the set $E_{G-A,G-A}$ of edges with both endpoints in $G-A$;
$$|E_{G-A,G-A}| \geq n- \left(\frac{4n-8}{k}\right)^2 - \left(\frac{8n-16}{k}-4\right)$$
which, together with $G$ having $2n-4$ edges, immediately implies the result.

Let us point out that
$$|A| \leq \frac{4n-8}{k}\ (\star)\qquad \mbox{and} \qquad |E_{A,A}| \leq 2|A|-4\ (\star\star),$$
where $E_{A,A}$ denotes the set of edges with both endpoints in~$A$, $(\star)$ follows from the definition of~$A$ and $(\star\star)$ from the subgraph of~$G$ induced by~$A$ being bipartite and planar.

We now consider a plane embedding of~$G$. Each face of~$G$ is bounded by four edges. For any vertex $v \in A$, removing~$v$ and its incident edges from the quadrangulation~$G$ gives a face bounded by a cycle of $2\deg(v) \geq 2k$ edges. Because of~$(\star)$, at most $\frac{4n-8}{k}$ vertices of this cycle belong to~$A$. We thus find at least $2\deg(v)- 2\frac{4n-8}{k}$ edges of this cycle that belong to $E_{G-A,G-A}$.

Repeating this counting process for each vertex of $A$ leads to at least
$$2\sum_{v \in A}  \deg{(v)}  - 2\left(\frac{4n-8}{k}\right)^2$$
edges for~$E_{G-A,G-A}$, but an edge $e$ of~$E_{G-A,G-A}$  might be counted up to four times, once for each high-degree vertex incident to a face of $G$ containing $e$ (see Figure~\ref{fig:graugran4aresta}).
\begin{figure}[h]
	\centering
		\includegraphics[width=0.6\textwidth]{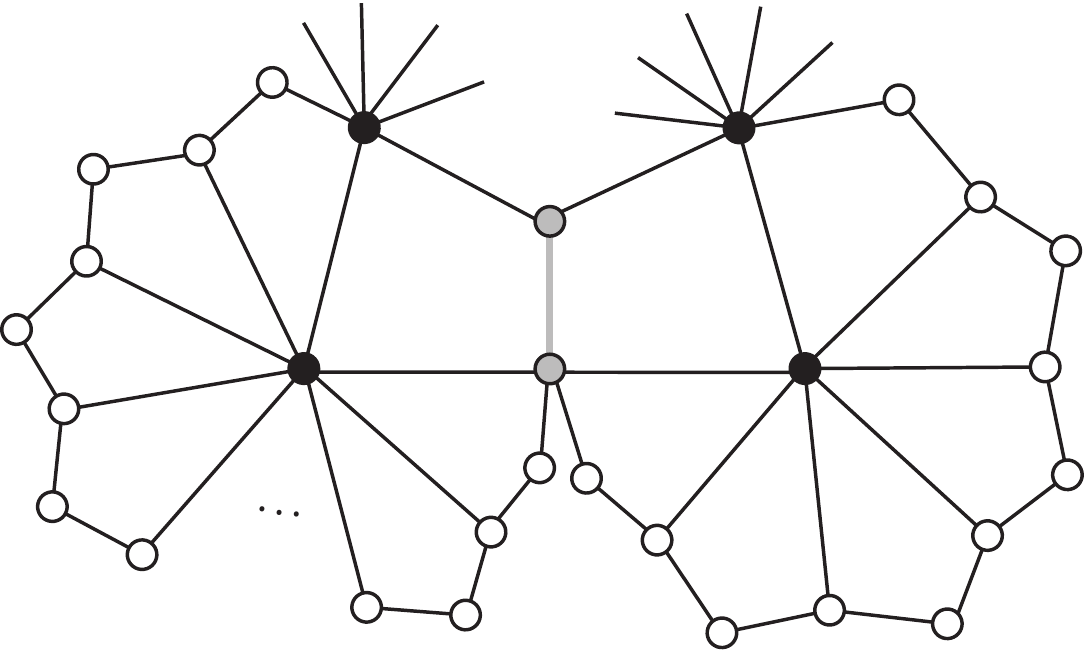}
	\caption{An edge (gray) is counted for up to four high-degree vertices (black).}
	\label{fig:graugran4aresta}
\end{figure}

If $e$ is counted three times, then it is counted for two adjacent vertices $x$ and $y$ of~$A$ and we can charge~$e$ to the edge $xy$ of~$E_{A,A}$. Likewise, if $e$ is counted four times, we charge it to two edges of~$E_{A,A}$. Observe that an edge $e$ of $E_{G-A,G-A}$ is charged to at most two edges of $E_{A,A}$. Also observe that an edge of $E_{A,A}$ gets a charge of at most two edges of $E_{G-A,G-A}$, because no edge is incident to more than two faces.

By~$(\star\star)$ we conclude that, in the counting process above, at most $2(2|A|-4)$ edges of~$E_{G-A,G-A}$ have been counted more than twice. Hence, $$|E_{G-A,G-A}| \geq \frac{1}{2}\left(2\sum_{v \in A}  \deg{(v)} - 2\left(\frac{4n-8}{k}\right)^2 - 2\left(2\frac{4n-8}{k}-4\right)\right).$$ Since we were assuming $\sum_{v \in A} \deg{(v)} \geq n$, we get
$$ |E_{G-A,G-A}| \geq n- \left(\frac{4n-8}{k}\right)^2 - \left(\frac{8n-16}{k}-4\right)$$
as desired.
\end{proof}

\begin{theorem}
\label{thm:upper_bipartite}
For the class of bipartite graphs with minimum vertex-degree $3$, $\lambda_{2\max}\leq1+O(\frac{1}{n^{1/3}}).$
\end{theorem}

\begin{proof}
Since $\lambda_2$ does not decrease when adding edges to a graph~\cite{fiedler}, we can assume that the graphs considered have the maximum number of edges. Hence, let $G$ be a maximal bipartite planar graph with~$n$ vertices and minimum vertex-degree~$3$.

We can assume that $G$ has a vertex of degree at least~$\frac{n}{8}$, since otherwise Theorem~\ref{thm:spielman-teng} directly gives $\lambda_2 \leq 1$. The proof is now analogous to that of Case~$1$ in Theorem~\ref{thm:upper_planar}: Let $A$ be the set of vertices of~$G$ with degree at least~$k$, then place the vertices of~$A$ at the origin and those of~$G-A$ on the unit sphere by using the embedding of Theorem~\ref{thm:spielman-teng}. Two cases arise:\\

\noindent\textit{Case 1:} {$\sum_{v \in A} \deg{(v)} \leq n$}.
Then the Embedding Lemma~\ref{lem:embedding} gives
$$\lambda_2(G) \leq \frac{8k}{n-|A|}+\frac{n}{n-|A|} \leq \frac{8k+n}{n-\frac{4n-8}{k}}
= \frac{8k^2 + n k}{n k-4n+8}$$
and choosing $k=\sqrt{n}$ we obtain $\lambda_2 \leq 1+O\left(\frac{1}{\sqrt{n}}\right).$

\noindent\textit{Case 2:} {$\sum_{v \in A} \deg{(v)} > n$}.
Using Lemma~\ref{lem:bipartiteEdges} we get that
$$|E_{A,G-A}| \leq  n + \left(\frac{4n-8}{k}\right)^2 + \frac{8n-16}{k}-8.$$
The Embedding Lemma~\ref{lem:embedding} then gives
$$\lambda_2(G) \leq \frac{8k}{n-|A|}+\frac{n + \left(\frac{4n-8}{k}\right)^2 + \frac{8n-16}{k}-8}{n-|A|}$$
$$\leq \frac{8k}{n-\frac{4n-8}{k}}+\frac{n + \left(\frac{4n-8}{k}\right)^2 + \frac{8n-16}{k}-8}{n-\frac{4n-8}{k}}$$
$$= \frac{8k^2+nk + \frac{(4n-8)^2}{k} + 8n-16 -8k}{nk-4n+8}.$$
By choosing $k=n^{2/3}$
we get $\lambda_2 \leq 1+O(\frac{1}{n^{1/3}}).$
\end{proof}

Note that the condition on the minimal degree is essential. Indeed, whenever a graph $G$ has a vertex of degree $2$, then $\lambda_2(G) \leq 2$, because the vertex-connectivity is an upper bound on $\lambda_2$~\cite{fiedler}. Also note that the complete bipartite graph $K_{2,n}$ is planar and $\lambda_2(K_{2,n}) =2$. Therefore we obtain the following corollary for bipartite planar graphs and $n$ large:
\begin{corollary}
For the class of bipartite planar graphs and for $n$ large, $\lambda_{2\max}=2.$
\end{corollary}

\subsection{Outerplanar graphs}

The following well-known lemma will allow us to use Theorem~\ref{thm:separator}.

\begin{lemma}\label{lem:tree}
Every tree $T$ on $n$ vertices contains a separator $X$ consisting of only one vertex such that each connected component of $T-X$ has at most $\frac{n}{2}$ vertices.
\end{lemma}


\begin{theorem}
\label{thm:upper_outerplanar}
For the class of outerplanar graphs, $\lambda_{2\max}\leq 1+O(\frac{1}{n})$.
\end{theorem}

\begin{proof}
Let $G$ be an outerplanar graph. Again, we can assume that $G$ has the maximum possible number of edges, since $\lambda_2$ does not decrease when adding edges to a graph~\cite{fiedler}.

Such a graph has a plane drawing as a subdivision of a convex polygon into triangles. Consider then the dual graph of this drawing of $G$, in which there is a vertex for each interior face of $G$ and two vertices of the dual graph are adjacent if two faces of $G$ share an edge. This dual graph is a binary tree $T$ with $n-2$ vertices. By Lemma~\ref{lem:tree}, $T$ contains a vertex $v$, such that $T-\{v\}$ is split into two or three components with at most $\frac{n-2}{2}$ vertices each.  Let $a,b$ and $c$ denote the three vertices of the face of $G$ corresponding to $v \in T$. The following cases arise:\\

\noindent\textit{Case 1:} A component of $T-\{v\}$ contains exactly $\frac{n-2}{2}$ vertices.\\
We can assume that $ab$ is the edge in~$G$ crossed by the edge of $T$ connecting $v$ with this large component of $T-\{v\}$. Then, we take $X=\{a,b\}$ as a separator for $G$. Both components of $G-X$ have $\frac{n-2}{2}$ vertices. Since $G$ is outerplanar, it contains no~$K_{2,3}$ and, therefore, at most two vertices of $G-X$ are adjacent to both $a$ and~$b$. It follows that the number of edges connecting $X$ to~$G-X$ is at most~$n$.
Hence, Theorem~\ref{thm:separator} gives $\lambda_2(G) \leq \frac{n}{n-2} = 1+O(\frac{1}{n}).$\\

\noindent\textit{Case 2:} No component of $T-\{v\}$ contains exactly $\frac{n-2}{2}$ vertices.\\
Then we take $X=\{a,b,c\}$ as a separator for $G$. Each component of $G-X$ has at most $\frac{n-3}{2}$ vertices.
Since $G$ is outerplanar, it contains no~$K_4$ nor~$K_{2,3}$.
Therefore, no vertex of $G-X$ is adjacent to all three vertices $a,b,c$ and for any two vertices $x,y\in\{a,b,c\}$ at most one vertex of $G-X$ is adjacent to both $x$ and~$y$. Consequently, the number of edges connecting $X$ with $G-X$ is at most $n$.
Hence, Theorem~\ref{thm:separator} in this case gives $\lambda_2(G) \leq \frac{n}{n-3} = 1+O(\frac{1}{n}).$
\end{proof}

\subsection{Graphs of bounded genus}\label{sec:genus}

For the class of graphs of bounded genus, the following result on the Fiedler value is known:

\begin{theorem}[Kelner~\cite{kelner06}]\label{thm:kelner06}
Let $G$ be a graph of genus $g$ on $n$ vertices and with bounded degree. Then, the Fiedler value of $G$ satisfies
$\lambda_2(G) \leq O(\frac{g}{n})$.
\end{theorem}

Moreover, inspecting the proof of~\cite{kelner06}, after an embedding on a surface of genus~$g$, and a suitable analytic mapping,
there exists an embedding of the vertices on the unit sphere, whose coordinates sum up to~$0$, attaining the bound $\lambda_2(G) \leq O(\frac{g}{n})$. Thus, we could use the Embedding Lemma~\ref{lem:embedding} to bound the Fiedler value of a bounded genus graph $G$ by placing vertices of high degree in the origin and using the embedding of ~\cite{kelner06} for the induced subgraph of vertices of bounded degree. This gives a bound $\lambda_{2\max} \leq 2+o(1)$ for the class of genus $g$ graphs~\cite{eurocomb}. Instead, we can go a bit further using Theorem~\ref{thm:separator} together with the following separator theorem for genus~$g$ graphs.

\begin{theorem}[Gilbert-Hutchinson-Tarjan~\cite{gilbert}]\label{thm:gilbert}
A graph of genus $g$ with $n$ vertices has a set of at most $6\sqrt{gn}+2\sqrt{2n}+1$ vertices whose removal leaves no component with more than $2n/3$ vertices.
\end{theorem}

\begin{theorem}
\label{thm:upper_genus2}
For the class of genus $g$ graphs, $\lambda_{2\max} \leq 2+O(\frac{1}{\sqrt{n}}).$
\end{theorem}
\begin{proof}
Let $G=(V,E)$ be a graph of genus $g$ with $n$ vertices. It is well known that $|E| \leq 3n-3\chi$, where $\chi$ is the Euler characteristic of $G$ and $\chi=2-2g$.

Furthermore we claim that, by Theorem~\ref{thm:gilbert}, any graph of genus $g$ has a separator $X$ of size at most $O(\sqrt{g}\sqrt{n})$ such that each connected component of $G-X$ has size at most $\frac{n-|X|}{2}=\frac{n-O(\sqrt{g}\sqrt{n})}{2}$ vertices. This follows from the fact that if a connected component has a greater number of vertices after application of Theorem~\ref{thm:gilbert}, then we can apply this theorem again on this component.

We now consider the bipartite graph with bipartition classes $X$ and~$G-X$. The genus of this bipartite graph is also at most~$g$, and its number of edges $|E_{X,G-X}|$ is at most~$2n+4g-4$ (that of a quadrangulation with $n-2+2g$ faces). Thus, by Theorem~\ref{thm:separator}, we get $\lambda_{2}(G) \leq \frac{|E_{X,G-X}|}{n-|X|}\leq 2+O(\frac{1}{\sqrt{n}}).$
\end{proof}

\subsection{$K_h$-minor-free graphs}
\label{sec:upper_linklessly}

Consider the class of $K_h$-minor-free graphs on $n$ vertices. Biswal et al.~\cite{BLR10} proved that for bounded-degree $K_h$-minor-free graphs, $\lambda_2 \in O\left(\frac{\Delta h^6 \log{h}}{n}\right)$, where $\Delta$ is the maximum degree. Using the maximum number of edges in a $K_h$-minor-free graph, we derive an upper bound on $\lambda_{2\max}$ which does not depend on~$\Delta$.

For $h \leq 9$, a $K_h$-minor-free graph with $n$ vertices that is $(h-2)$-connected has less than $(h-2)n$ edges~\cite{song}. For large values of~$h$, a $K_h$-minor-free graph with $n$ vertices has at most $\alpha nh\sqrt{\log(h)}$ edges~\cite{thomason}, for $\alpha=0.319\ldots + o(1)$. We will show that $\lambda_{2\max}$ is at most this number of edges, divided by~$n$, plus a small error term.

The proof relies on the following separator theorem for non-planar graphs~\cite{alon}.

\begin{theorem}[Alon-Seymour-Thomas~\cite{alon}]
\label{thm:separator2}
Let $G$ be a graph with $n$ vertices and with no $K_h$-minor. Then $G$ has a separator $X$ of order $O(h^{3/2}\sqrt{n})$ such that each connected component of $G-X$ has at most $2n/3$ vertices.
\end{theorem}

We remark that it has recently been shown that $K_h$-minor-free graphs even have a separator of size $O(h\sqrt{n})$~\cite{reed}.

\begin{theorem}
For the class of $K_h$-minor-free graphs,  and for $h\leq 9$, $$\lambda_{2\max} \leq h-2+O\left(\frac{1}{\sqrt{n}}\right)$$
and for large values of $h$
$$\lambda_{2\max} \leq \alpha h\sqrt{\log(h)} + O\left({\frac{ \alpha h^{5/2}\sqrt{\log(h)}}{\sqrt{n}}}\right),$$ where \ $\alpha = 0.319\ldots + o(1)$.
\end{theorem}

\begin{proof}
Let $G=(V,E)$ be a $K_h$-minor-free graph with $n$ vertices.
By Theorem~\ref{thm:separator2}, $G$ has a separator $X'$ of size $O(h^{3/2}\sqrt{n})$ such that each connected component of $G-X'$ has at most $2n/3$ vertices. At most two connected components of $G-X'$ have more than $\frac{n-O(h^{3/2}\sqrt{n})}{2}$ vertices. We can again apply Theorem~\ref{thm:separator2} to these two components to obtain a separator $X$ of $G$ of size $O(h^{3/2}\sqrt{n})$ such that each connected component of $G-X$ has at most $\frac{n-|X|}{2}$ vertices. We now use Theorem~\ref{thm:separator}.

For $h\leq 9$ we can assume that $G$ is $(h-2)$-connected, as otherwise $\lambda_2(G)$ is at most $h-2$~\cite{fiedler}. Then $G$ has at most $(h-2)n$ edges. We obtain
$$\lambda_2(G) \leq  \frac{(h-2)n} {n-O(h^{3/2}\sqrt{n})} = h-2 + O\left(\frac{1}{\sqrt{n}}\right).$$
For large values of $h$ we obtain
$$\lambda_2(G) \leq  {\frac{\alpha hn\sqrt{\log(h)}} {n-O(h^{3/2}\sqrt{n})}} = \alpha h\sqrt{\log(h)} + O\left({\frac{ \alpha h^{5/2}\sqrt{\log(h)}}{\sqrt{n}}}\right).$$
\end{proof}

\section{Lower bounds on $\lambda_{2\max}$ }
\label{sec:lower_planar}



\begin{figure}[htb]
\begin{center}
\includegraphics[width=\textwidth]{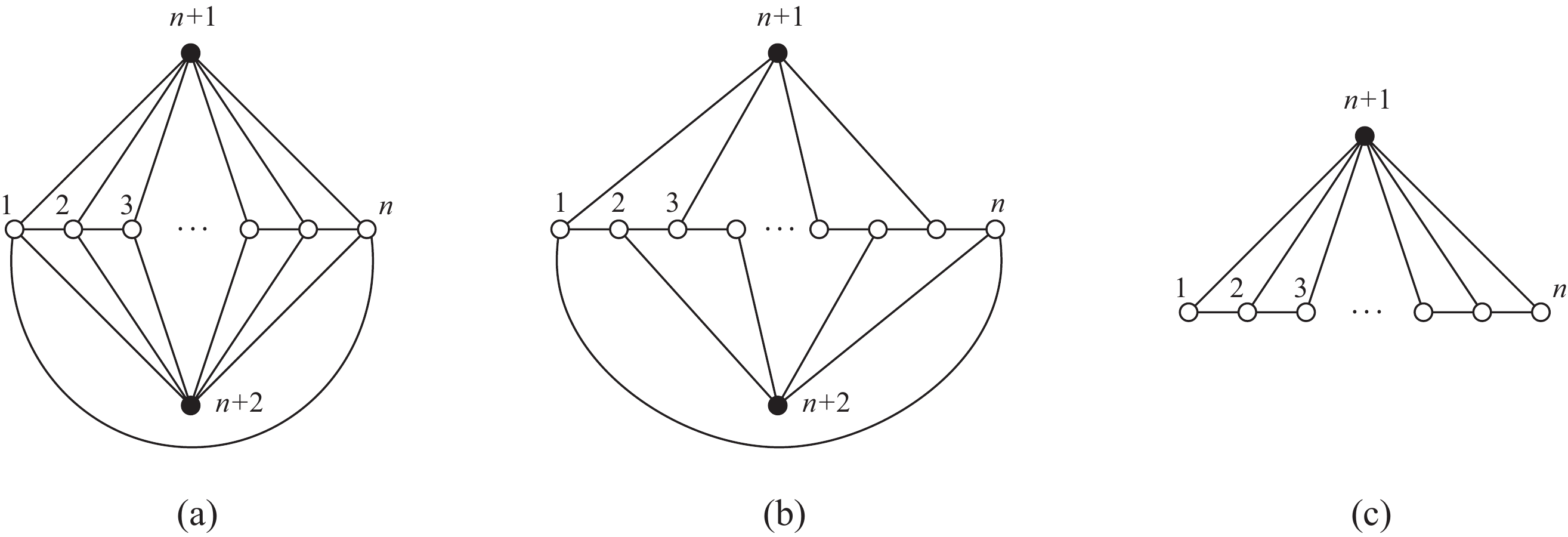}
\caption{(a) The doublewheel graph $C_n\ast 2K_1$. (b) Quadrangulation. (c) Outerplanar graph.}
\label{fig:lower}
\end{center}
\end{figure}

\subsection{General planar graphs and graphs of bounded genus}

The lower bound example on $\lambda_{2\max}$ for the class of planar graphs is
the graph $C_n\ast 2K_1$, the join of a cycle and two isolated vertices, see Figure~\ref{fig:lower}~(a).

\begin{lemma}[see also~\cite{freitas}]
\label{lemma:doublewheel}
Assume $n\ge 4$ even. Let $G=C_n\ast 2K_1$
be the join of an $n$-cycle and two isolated vertices.
 Then,
$ \lambda_2(G)=4-2\cos(2 \pi/n)=2+\Theta(1/n^2).$
\end{lemma}
\begin{proof}
It is well known (see e.g.~\cite{mohar91}) that
the eigenvalues of $C_n$ are $0$ (with multiplicity $1$),
$2-2 \cos (2 \pi k/n)$ for $k=1,\ldots,\frac{n}{2}$ (with multiplicity $2$ for $k=1,\ldots,\frac{n}{2}-1$, and multiplicity $1$ for $k=\frac{n}{2}$).
The only eigenvalue of $2K_1$ is $0$ (with multiplicity $2$).
By Corollary~3.7 in~\cite{mohar91} we get that the Laplacian polynomial of~$G$ is
\begin{equation}\label{eq:triangulations}
\mu(G,x)=\frac{x(x-n-2)}{(x-n)(x-2)}	(x-n)^2 (x-2)(x-6)\prod_{k=1}^{\frac{n}{2}-1}(x-4+2\cos(2 \pi k/n))^2.
\end{equation}
Hence, the eigenvalues of $G$ are $0,6,n,n+2$ (all with multiplicity $1$), and $4- 2\cos(2 \pi k/n)$ for $k=1,\ldots,\frac{n}{2}-1$ (all with multiplicity $2$).
\end{proof}

This same family of graphs provides also the lower bound example for graphs of bounded genus.

\subsection{Bipartite planar graphs}

\begin{lemma}
Assume $n\ge 4$, even. Let $G$ be the quadrangulation on $n+2$ vertices
obtained from the $n$-cycle $C_n$ by connecting $\frac{n}{2}$ independent vertices to a new vertex
and the remaining $\frac{n}{2}$ (also independent) vertices to another new vertex,
see Figure~\ref{fig:lower}~(b). Then,
$ \lambda_2(G)=3-2\cos(2 \pi/n)=1+\Theta(1/n^2).$
\end{lemma}
\begin{proof}
We explicitly give all eigenvectors and eigenvalues.
Assume that the vertices in the cycle are labeled from $1$ to $n$, in cyclic order.
Assume also that the vertex of degree $n/2$ adjacent only to vertices of odd label is labeled $n+1$,
and the vertex of degree $n/2$ adjacent only to vertices with even label is labeled $n+2$.

By direct calculation one can check that
the vectors
$$(x_k)_i=\sin\left(\frac{2 \pi i k}{n}\right),\ \mbox{for }i\in\{1,\ldots,n\}\quad
\mbox{and}
\quad (x_k)_i=0,\ \mbox{for } i\in\{n+1,n+2\},$$
as well as
$$(y_k)_i=\cos\left(\frac{2 \pi i k}{n}\right),\ \mbox{for } i\in\{1,\ldots,n\}\quad
\mbox{and}
\quad (y_k)_i=0,\ \mbox{for } i\in\{n+1,n+2\},$$
are eigenvectors associated to the eigenvalue $3-2\cos\frac{2 \pi k}{n}$, for any $k\in\{1,\ldots,\frac{n}{2}-1\}$.

The other eigenvectors and corresponding eigenvalues are as follows.
The all-ones vector corresponds to $\lambda_1=0$.
The vector $v$ with
$$v_i=1,\ \mbox{for } i\in\{1,\ldots,n\}\quad
\mbox{and}
\quad v_{n+1}=v_{n+2}=-\frac{n}{2}$$
is an eigenvector with corresponding eigenvalue $\frac{n}{2}+1$.
The vector $u$ defined as
$$u_i=1,\ \mbox{for } i\ \mbox{odd and } i \leq n,
\quad
u_i=-1,\ \mbox{for } i\ \mbox{even and } i \leq n,$$
$$u_{n+1}=-u_{n+2}=\frac{5}{2}-\frac{n}{4}+\frac12\sqrt{n^2/4-7n+25}$$
is an eigenvector with corresponding eigenvalue $5-u_{n+1} > \frac{5}{2}$.
Finally the vector $w$ defined as
$$w_i=u_i,\ \mbox{for } i \leq n,
\quad
w_{n+1}=-w_{n+2}=\frac{5}{2}-\frac{n}{4}-\frac12\sqrt{n^2/4-7n+25}$$
is an eigenvector with corresponding  eigenvalue $5-w_{n+1} > \frac{5}{2}$.

Therefore, $\lambda_2(G) = 3-2 \cos \left(\frac{2 \pi}{n}\right)=1+\Theta(1/n^2).$
\end{proof}

\subsection{Outerplanar graphs}

\begin{lemma}
\label{lemma:fan}
Assume $n\ge 2$. Let $G=P_n\ast K_1$ be the fan graph, obtained by the join of the path $P_n$ with $n$ vertices and the isolated vertex $K_1$, see Figure~\ref{fig:lower}~(c). Then,
$ \lambda_2(G)=3-2\cos( \pi/n)=1+\Theta(1/n^2).$
\end{lemma}

\begin{proof}
It is well known (see, e.g.~\cite{brouwer}) that
the eigenvalues of $P_n$ are $2- 2\cos(\frac{\pi k}{n})$, for $k\in\{0,..,n-1\}$.
The only eigenvalue of $K_1$ is $0$.
Then, by Corollary~3.7 in~\cite{mohar91},
\begin{equation}\label{eq:outerplanar}
\mu(G,x)=\frac{x(x-n-1)}{(x-n)(x-1)}	(x-n)\prod_{k=0}^{n-1}(x-3+2\cos({ \pi k}/{n})).
\end{equation}
Therefore the eigenvalues of $G$ are $0, 3-2\cos(\frac{ \pi k}{n})$, for $k\in\{1,\cdots, n-1\}$, and $n+1$. Thus $\lambda_2(G)=3-2\cos\left(\frac{ \pi }{n}\right) = 1+\Theta(\frac{1}{n^2}).$
\end{proof}

\subsection{$K_h$-minor-free graphs}

Let $n\geq 2h-4.$ The complete bipartite graph $K_{{h-2},{n-h+2}}$ is $K_h$-minor-free and $\lambda_2(K_{{h-2},{n-h+2}}) = h-2$~\cite{abreu}.

\section{Concluding remarks}

We have shown bounds on the maximum Fiedler value $\lambda_{2\max}$, detailed in Table~\ref{tab:results}, for several classes of graphs. Two different techniques have been presented, both based on the Embedding Lemma~\ref{lem:embedding} together with a second tool: On one hand, the embedding of Spielman and Teng for planar graphs. On the other hand, an appropriate separator.
In either method, the quality of the bound on the Fiedler value depends on the estimation of the number of edges connecting vertices of high degree with others of low degree, in the first case, or the number of edges incident to a separator, in the second case.

We expect that a similar approach works for the class of linklessly embeddable graphs, where it would be interesting to see $\lambda_{2\max}=3+o(1)$.

Needless to say, it would be interesting to close the gap between the lower and upper bounds we obtained.
We conjecture that the constructions presented provide, indeed, the correct values for the corresponding graph classes. In particular, for planar graphs and $n=6$, the graph $K_{2,2,2}$ which maximizes $\lambda_{2\max}$ for $n=6$ is a special case of the lower bound example presented here. Analogously, for bipartite planar graphs with minimum vertex degree~$3$, the maximizing example for $n=8$ is the cube, which has $\lambda_2=2$, and also belongs to the presented family of graphs.

During the revision phase a new paper~\cite{NO} appeared. There the authors reprove our results for graphs of genus $g$ and for $K_h$-minor-free graphs, for which they extend the upper bound given here for $h\leq 9$ to all values of $h$. Their results come from a general statement using shallow topological minors and they also use some of our results from~\cite{eurocomb}. For planar graphs, however, the upper bound they obtain is weaker than the one obtained in this paper.

\section{Acknowledgments}

L.~Barri\`ere was partially supported by Ministerio de Ciencia e Innovaci\'{o}n, Spain, and the European Regional Development Fund under project MTM2011-28800-C02-01 and by the Catalan Research Council under project 2009SGR1387. C.~Huemer was partially supported by projects MTM2009-07242, Gen. Cat. DGR 2009SGR1040, and ESF EUROCORES programme EuroGIGA, CRP ComPoSe: MICINN Project EUI-EURC-2011-4306. D.~Mitsche gratefully acknowledges support from NSERC and MPrime. D.~Orden was partially supported by MICINN grants MTM2008-04699-C03-02 and MTM2011-22792, and by the ESF EUROCORES programme EuroGIGA, CRP ComPoSe, under grant EUI-EURC-2011-4306.


\begin{thebibliography}{99}

\bibitem{abreu}
Abreu, N. M. M. de,
\newblock {\em Old and new results on algebraic connectivity of graphs},
\newblock Linear Algebra and its Applications {\bf 423} (2007), 53--73.

\bibitem{alon}
Alon, N., Seymour, P., and Thomas, R.,
\newblock {\em A separator theorem for nonplanar graphs},
\newblock Journal of the AMS {\bf 3(4)} (1990),  801--808.


\bibitem{eurocomb}
Barri\`ere, L., Huemer, C., Mitsche D., and Orden, D.,
\newblock {\em On the Fiedler value of large planar graphs},
\newblock EUROCOMB'11, Electronic Notes in Discrete Mathematics {\bf{38}} (2011), 111--116.


\bibitem{BLR10}
Biswal, P., Lee, J. R., and Rao, S.,
\newblock {\em Eigenvalue bounds, spectral partitioning, and metrical deformations via flows},
\newblock Journal of the ACM {\bf 57(3)} (2010),  13:1--13:23. DOI: 10.1145/1706591.1706593.

\bibitem{brouwer}
Brouwer, A.E., and Haemers, W. H.,
\newblock {\em Spectra of graphs}, Springer, 2012. ISBN 978-1-4614-1938-9.

\bibitem{elsner}
Elsner, U.,
\newblock {\em Graph partitioning. A survey},
\newblock preprint SFB393/97-27, Technische Universit\"at Chemnitz (1997).

\bibitem{fiedler}
Fiedler, M.,
\newblock {\em Algebraic connectivity of graphs},
\newblock Czechoslovak Math. J. {\bf  23(98)} (1973), 298--305.

\bibitem{freitas}
Freitas, P.,
\newblock {\em A Heawood-type result for the algebraic connectivity of graphs on surfaces},
\newblock preprint (2001). URL: \url{http://arxiv.org/abs/math/0109191v1}.

\bibitem{gilbert}
Gilbert, J.R.,  Hutchinson, J.P., and  Tarjan, R.E.,
\newblock{\em A separator theorem for graphs of bounded genus},
\newblock  Journal of Algorithms {\bf{5(3)}} (1984), 391--407.

\bibitem{reed}
Kawarabayashi, K.,  and Reed, B. A.,
\newblock {\em A Separator Theorem in Minor-Closed Classes},
\newblock in ``FOCS '10. 51th Annual IEEE Symposium on Foundations of Computer Science", 153--162, 2010.

\bibitem{kelner06}
Kelner, J. A.,
\newblock {\em Spectral partitioning, eigenvalue bounds, and circle packings for graphs of bounded genus},
\newblock SIAM J. COMPUT. {\bf 35(4)} (2006),  882--902.

\bibitem{KLPT09}
Kelner, J. A., Lee, J. R., Price, G. N., and Teng, S.-H.,
\newblock {\em Higher eigenvalues of graphs},
\newblock in ``FOCS '09. 50th Annual IEEE Symposium on Foundations of Computer Science", 735--744, 2009.

\bibitem{koebe}
Koebe, P.,
\newblock {\em Kontaktprobleme der Konformen Abbildung},
\newblock {Ber. S\"achs. Akad. Wiss.} Leipzig, Math.-Phys. {\bf 88} (1936), 141--164.

\bibitem{merris}
Merris, R.,
\newblock  {\em Characteristic Vertices of Trees},
\newblock Linear and Multilinear Algebra {\bf 22} (1987), 115--131.

\bibitem{mohar91}
Mohar, B.,
\newblock {\em The Laplacian spectrum of graphs},
\newblock in Y. Alavi, G. Chartrand, O. R. Oellermann, A. J. Schwenk, eds., ``Graph Theory, Combinatorics, and Applications (Vol. 2)", Wiley, 871--898, 1991.

\bibitem{mohar}
Mohar, B.,
\newblock {\em Some applications of Laplace eigenvalues of graphs},
\newblock in G. Hahn and G.Sabidussi eds., ``Graph Symmetry: Algebraic Methods and Applications", NATO ASI Series C, vol. 497, 225--275, 1997.

\bibitem{molitierno}
Molitierno, J. J.,
\newblock {\em On the algebraic connectivity of graphs as a function of genus},
\newblock Linear Algebra and its Applications {\bf 419} (2006), 519--531.


\bibitem{NO}
Ne\v{s}et\v{r}il, J., Ossona de Mendez, P.,
\newblock {\em  A note on Fiedler values of classes with sublinear separators}.
\newblock URL: \url{http://arxiv.org/abs/1208.3581}.


\bibitem{ST96}
Spielman, D., and Teng, S.-H.,
\newblock Spectral partitioning works: Planar graphs and finite element meshes,
\newblock Tech. Report: EECS Department, UC Berkeley (1996).
URL:
\url{http://www.eecs.berkeley.edu/Pubs/TechRpts/1996/5359.html}.

\bibitem{ST07}
Spielman, D., and Teng, S.-H.,
\newblock {\em Spectral partitioning works: Planar graphs and finite element meshes},
\newblock { Linear Algebra and its Applications} {\bf 421(2-3)} (2007), 284--305.

\bibitem{song}
Song, Z.-X.,  and Thomas, R.,
\newblock {\em The extremal function for $K_9$ minors},
\newblock Journal of Combinatorial Theory, Ser. B {\bf 96(5)} (2006), 240--252.

\bibitem{thomason}
Thomason, A.,
\newblock {\em The Extremal Function for Complete Minors},
\newblock {Journal of Combinatorial Theory, Ser. B} {\bf 81(2)} (2001), 318--338.




\end{thebibliography}
\end{document}